\documentclass[12pt,reqno]{amsart}
\usepackage{mathrsfs}
\usepackage[breaklinks]{hyperref}
\usepackage[headheight=110pt,top=1.25in, bottom=1.25in, left=1in, right=1in]{geometry}

\usepackage{url}
\usepackage{amssymb}
\usepackage{amsmath}
\allowdisplaybreaks

\renewcommand{\Re}{\operatorname{Re}}

\newcommand{\g}{\gamma}

\newcommand{\z}{\zeta}

\renewcommand{\[}{\left\[}
\renewcommand{\]}{\right\]}
\newcommand{\Z}{\mathbb{Z}}
\newcommand{\K}{\mathbb{K}}

\newcommand{\C}{\mathbb{C}}

\renewcommand{\ge}{\geqslant}

\numberwithin{equation}{section}
\theoremstyle{plain}
\newtheorem{theorem}{Theorem}[section]
\newtheorem{lemma}[theorem]{Lemma}
\newtheorem{corollary}[theorem]{Corollary}
\newtheorem{proposition}[theorem]{Proposition}

\makeatletter
\def\proof{\@ifnextchar[{\@oproof}{\@nproof}}
\def\@oproof[#1][#2]{\trivlist\item[\hskip\labelsep\textit{#2 Proof of\
		#1.}~]\ignorespaces}
\def\@nproof{\trivlist\item[\hskip\labelsep\textit{Proof.}~]\ignorespaces}

\makeatother

\makeatletter
\def\@tocline#1#2#3#4#5#6#7{\relax
	\ifnum #1>\c@tocdepth 
	\else
	\par \addpenalty\@secpenalty\addvspace{#2}%
	\begingroup \hyphenpenalty\@M
	\@ifempty{#4}{%
		\@tempdima\csname r@tocindent\number#1\endcsname\relax
	}{%
		\@tempdima#4\relax
	}%
	\parindent\z@ \leftskip#3\relax \advance\leftskip\@tempdima\relax
	\rightskip\@pnumwidth plus4em \parfillskip-\@pnumwidth
	#5\leavevmode\hskip-\@tempdima
	\ifcase #1
	\or\or \hskip 1em \or \hskip 2em \else \hskip 3em \fi%
	#6\nobreak\relax
	\dotfill\hbox to\@pnumwidth{\@tocpagenum{#7}}\par
	\nobreak
	\endgroup
	\fi}
\makeatother

\usepackage{bigints}
\usepackage{suffix}
\usepackage{mathtools}
\DeclarePairedDelimiterX\MeijerM[3]{\lparen}{\rparen}%
{\begin{smallmatrix}#1 \\ #2\end{smallmatrix}\delimsize\vert\,#3}

\newcommand\MeijerG[8][]{%
	G^{\,#2,#3}_{#4,#5}\MeijerM[#1]{#6}{#7}{#8}}

\WithSuffix\newcommand\MeijerG*[7]{%
	G^{\,#1,#2}_{#3,#4}\MeijerM*{#5}{#6}{#7}}

\usepackage{color}
\usepackage{amsmath}

\definecolor{blue}{rgb}{0,0,1}
\definecolor{red}{rgb}{1,0,0}
\definecolor{green}{rgb}{0,.6,.2}
\definecolor{purple}{rgb}{1,0,1}

\numberwithin{theorem}{section}
\numberwithin{equation}{section}

\begin{document}
	\title[Hecke-type action on Higher order Herglotz-Zagier function]{Hecke-type action on Higher order Herglotz-Zagier function}
	\author{Soumyarup Banerjee and Riya Mandal}
	\address{Department of Mathematics, Indian Institute of Technology Kharagpur, Kharagpur, Midnapore - 721302, West Bengal, India.}
	\email{soumyarup@maths.iitkgp.ac.in}
	
	\address{Department of Mathematics, Indian Institute of Technology Kharagpur, Kharagpur, Midnapore - 721302, West Bengal, India.}
	\email{riyamandaljkc@gmail.com}
	
\begin{abstract}
In a seminal paper, Lewis and Zagier constructed variety of functions satisfying the three-term functional equations. In this article, we consider the first example among them and establish that the function is a Hecke eigen form with respect to the Hecke operators, which acts on periods. We then utilize this result to determine the action of the aforementioned operators on the derivative of the Higher order Herglotz-Zagier function. The action leads to a family of multi-term functional equations satisfied by the function. 
\end{abstract}
	\maketitle
	\vspace{-0.8cm}

\section{Introduction}\label{Intro}
The Kronecker limit formula occupies a prominent position among the celebrated results in number theory due to its wide ranging applications in diverse areas of mathematics, including geometry,  theoretical physics. For any number field $\K$, if $A$ denotes an ideal class from the ideal class group of $\K$, then the Dedekind zeta function $\zeta_\K(s)$ can be decomposed as $\zeta_\K(s)= \sum_A \zeta(s, A)$, where for $\Re(s)>1$, $\zeta(s, A) = \sum_{\mathrm{a}\in A} \frac{1}{\mathcal{N}(\mathrm{a})^s}$, with $\mathcal{N}(\mathrm{a})$ being the norm of the ideal $\mathrm{a}$. The constant term in the Laurent series of $\zeta(s, A)$ at the pole $s=1$ can be described by the Kronecker limit formula, which was named after Kronecker \cite{Kronecker} for his contribution in the case of an imaginary quadratic field.	
	
A seminal contribution \cite{Zagier} of Don Zagier concerning the Kronecker limit formula for real quadratic fields has fascinated number theorists for decades. For $\psi(s) = \frac{\Gamma'(s)}{\Gamma(s)}$ denoting the logarithmic derivative of the Euler gamma function, the  following infinite series
	\begin{align}
		F(x):= \sum_{n=1}^{\infty} \frac{\psi(nx) - \log(nx)}{n} \quad \qquad (x \in \C\setminus (-\infty, 0]), \nonumber
	\end{align}
plays a significant role in the Kronecker limit formula associated to real quadratic fields. The function similar to $F(x)$ had already been appeared in an earlier work of Herglotz \cite{Herglotz} and hence Radchenko and Zagier termed this function as Herglotz function (sometimes known as the Herglotz-Zagier function) in their article \cite{Radchenko}, where they have demonstrated important arithmetic properties of $F(x)$, for instance, its relation with the Dedekind eta-function, multi-term functional equations satisfied by $F(x)$ and its cohomological aspects etc. Prior to this, Zagier \cite[Equations (7.4), (7.8)]{Zagier} obtained the following two-term and three-term functional equations that the Herglotz function satisfies, namely, for $x \in \C\setminus (-\infty, 0]$
	\begin{align}
		&F(x)+F\left(\frac{1}{x}\right) = 2 F(1)+ \frac{\log^2 x}{2} - \frac{\pi^2 (x-1)^2}{6x},\nonumber\\
		&F(x)-F(x+1)-F\left(\frac{x}{x+1}\right)=-F(1)+\operatorname{Li}_2\left(\frac{1}{1+x}\right), \label{three term Herglotz}
	\end{align}
where 
$$\operatorname{Li}_2(t):= \sum_{n=1}^\infty \frac{t^n}{n^2} \qquad \quad (0<t<1),$$
 is the Euler's Dilogarithm function. The above functional equations possess significant role in the proof of Meyer’s theorem \cite{Zagier}, and in the asymptotic expansions \cite{Radchenko}  of $F(x)$  near $0$ and $1$. Functional equations analogous to \eqref{three term Herglotz} appear in diverse contexts such as period functions for Maass forms \cite[Equation  (0.1)]{Lewis}, cotangent functions, and double zeta functions etc. We refer the elegant survey article of Don Zagier \cite{ServeyZagier} for further studies on this topic. 
 
The general Laurent coefficient of $\zeta(s, A)$ at the pole $s=1$ was determined by Ishibashi \cite{Ish03}. In this connection, the function that naturally emerges is the $k$-th order generalization of the Herglotz-Zagier function, namely
		\begin{align}
		\Phi_k(x)=\sum_{n=1}^{\infty} \frac{k\psi_{k-1}(nx)-\log^k{(nx)}}{n} \quad\qquad (x \in \mathbb{C}\setminus (-\infty,0]),\nonumber
	\end{align} 
	where $\psi_k(x) = \frac{\Gamma_k'(x)}{\Gamma_k(x)} $ is the logarithmic derivative of the generalized gamma function $\Gamma_k(x)$, given by
	\begin{align*}
		\Gamma_k(x):= \lim_{n \to \infty} \frac{\exp{(\frac{\log^{k+1}(n)}{k+1}x)}\prod_{j=1}^{n}\exp{(\frac{\log^{k+1}(j)}{k+1})}}{\prod_{j=0}^{n}\exp{(\frac{\log^{k+1}(j+z)}{k+1})}},
	\end{align*}
	which was introduced by Dilcher in \cite{Dilcher}. The function $\psi_k(x)$ has significant importance in number theory as it is directly connected to the generalized Stieltjes constant $\gamma_k(x)$ via the relation $\psi_k(x) = - \gamma_k(x)$.
	
Recently, Dixit et al. established the following two-term functional equation satisfied by the higher order Herglotz-Zagier function $\Phi_k(x)$, valid for every non-negative integer $k$ and $x>0$,  namely \cite[p.~26, Theorem 3.4]{DixitMTZ} 
	  \begin{align}
	  	&\sum_{j=0}^{k}\frac{(-1)^{j+1}}{j+1} \binom{k}{j}\frac{\log^{\,k-j}(x)}{2^{\,k-j}}
	  	\Bigg(\Phi_{j+1}(x) + (-1)^{k-j}\Phi_{j+1}\!\left(\frac{1}{x}\right)
	  	\Bigg)\nonumber\\
	  	&\quad=\sum_{j=0}^{k} \binom{k}{j} \frac{\log^{\,k-j}(x)}{2^{\,k-j}} \Bigg[j L^*_{j-1}(x) - \frac{1}{2} a_{j+1,0} \log^2(x) + \left(x + \frac{(-1)^{k-j}}{x}-2\right) \z^{(j)}(2) \nonumber\\
	  	&\quad- j \sum_{\ell=1}^{j} \frac{a_{j,\ell-1}}{\ell}\bigg(\sum_{n=0}^{\ell-1} (-1)^{n+1} c_{\ell-n}\frac{(\ell+1)!}{(n+1)!} \log^{\,n+1}(x)+ \frac{(-1)^{\ell}}{\ell+2} \log^{\,\ell+2}(x)
	  	\bigg)+ \frac{2(-1)^{j+1}}{j+1} \Phi_{j+1}(1)\Bigg],\nonumber
	  \end{align}
 where
	  	 \begin{align}
	  	 c_k &:=\frac{1}{k!} \int_0^{\infty} \Biggl[\log^k(t) - \log^k\!\Bigl(\frac{t}{1-e^{-t}}\Bigr)\Biggr]
	  	\left(\frac{1}{e^t - 1} - \frac{1}{t}\right) dt, \nonumber\\
	  	L^*_{k}(x) &:=
	  	\begin{cases}
	  		\displaystyle \lim_{z \to 1} \frac{d^k}{dz^k} I(z,x) & \text{for~}k \in \mathbb{N} \cup \{0\} \\
	  		0 & \text{for~}k=-1,\nonumber
	  	\end{cases}		
	  \end{align}
with the integral
	  \begin{align}
	  	&I(z,x) :=\frac{1}{\Gamma(z)}\int_0^{\infty} t^{z-2} \log\!\left(\frac{1 - e^{-xt}}{1 - e^{-t}}\right)\log\!\left(\frac{1 - e^{-t}}{t}\right)\, dt. \nonumber
	  \end{align}
Here $a_{k,j}$ is the constant defined recursively by the relation
	  \begin{align*}
	  	a_{k,j}=-\sum_{r=0}^{k-2}\binom{k-1}{r}\Gamma^{k-r-1}(1)~a_{r+1,j} \quad \qquad (0\le j\le k-1),
	  \end{align*}
starting from $a_{1,0}=1$, and $a_{k,k-1}=1$, where $\Gamma^k(1)$ denotes the value of the $k$-th derivative of $\Gamma(s)$ with respect to $s$ at $s=1$. In the same article, the authors have also shown that the function $\Phi'_{k}(x)$ satisfies the following three-term functional equation for every non-negative integer $k$ and $x>1$, which is given by \cite[p. 29, Theorem 3.8]{DixitMTZ}
	 \begin{multline}
	 		\sum_{j=0}^{k}\binom{k}{j}\frac{(-1)^{j+1}}{j+1}\left(\frac{\log x}{2}\right)^{k-j}\Bigg(
	 	\Phi'_{j+1}(x)- \Phi'_{j+1}(x-1)+ \frac{1}{x^{2}}(-1)^{k-j}\Phi'_{j+1}\!\left(\frac{x-1}{x}\right)
	 	\Bigg)\nonumber\\
	 	=\sum_{j=0}^{k}\binom{k}{j}\left(\frac{\log x}{2}\right)^{k-j}	\sum_{\ell=0}^{j}\binom{j}{\ell}
	 	(-1)^{j}\gamma_{j-\ell}\Bigg(\frac{\log^{\ell} x}{x}-\frac{\log^{\ell}(x-1)}{x-1}+\frac{(-1)^{k-j}\log^{\ell}\!\left(\frac{x-1}{x}\right)}{x(x-1)}
	 	\Bigg)\nonumber\\
	 	+\frac{1}{(k+1)x}\left[\left(\frac{-\log x}{2}\right)^{k+1}-\left(\frac{\log x}{2} -\log(x-1)\right)^{k+1}\right],
	 \end{multline} 
where $\gamma_j$ denotes the $j$-th Stieltjes constant. 

In the present article, we determine a family of multi-term functional equations satisfied by $\Phi'_{k}(x)$. To this end, we investigate in different aspects, the action of the operators, which acts like the Hecke operators on the space of the period functions. 

We next briefly recall the definition of the period functions and describe the action of the Hecke operators on it. Let $\Gamma$ be the full modular group $SL_2(\mathbb{Z})$, which is generated by the matrices $T=\tiny{\begin{pmatrix}
		1 & 1\\
		0 & 1
	\end{pmatrix}}$ and $S=\tiny{\begin{pmatrix}
		0 & -1\\
		1 & 0
	\end{pmatrix}}$. We recall the action of the standard slash operator of weight $m$, which can be defined as 
\begin{align*}
(f|_{m}\gamma)(z):= \frac{(ad-bc)^{m/2}}{(cz+d)^m}f\left(\frac{az+b}{cz+d}\right),
\end{align*}
with $\gamma$ being the matrix $\gamma = \tiny{\begin{pmatrix}
		a & b\\
		c & d
	\end{pmatrix}}.$  

A rational period function of weight $2m$ over the full modular group is a rational function $q(z)$ that arises in the definition of a modular integral of weight $2m$ over $\Gamma$. A modular integral of weight $2m$ over $\Gamma$ is a meromorphic function $F$ defined on the upper half plane $\mathcal{H}$,  satisfying
	\begin{align}\label{modular_integral_relation}
		F|_{2m}T(z)=F(z) \qquad \text{and} \qquad F|_{2m}S(z)=F(z)+q(z).
	\end{align}
It follows from the transformations in \eqref{modular_integral_relation} along with the relation $S^2 = (TS)^3= I$ that the rational period function $q(z)$ satisfies 
	\begin{align}\label{rational period function relation}
		q|_{2m}S+q=0, \qquad q|_{2m}(TS)^2+q|_{2m}(TS)+q=0.
	\end{align}
The definition of the modular integral implies that every modular integral of weight $2m$ over $\Gamma$ has an associated rational period function of same weight over $\Gamma$. Knopp \cite{Knopp3} proved that the converse also holds. The rational functions satisfying the transformation laws \eqref{rational period function relation} are precisely the rational period function of weight $2m$ over $\Gamma$.

We next present the action of the Hecke operator on period functions and for that we fix few notations. For any positive integer $n$, let $\mathcal{M}_n$ be the set of $2\times 2$ integer matrices of determinant $n$, modulo $\{\pm 1\}$ and $\mathcal{R}_n=\mathbb{Q}[\mathcal{M}_n]$. Let $T_n^\infty$ be the classical Hecke operator, given by
		\begin{align*}
		T_n^\infty := \sum_{ad=n}\sum_{0\le b<d} \begin{bmatrix}
			a & b\\
			0 & d
		\end{bmatrix} \in \mathcal{R}_n,
	\end{align*}
which acts on the space of modular forms of weight $2m$ over $\Gamma$. Knopp \cite[p. 53]{Knopp1} introduced a Hecke operator $\widetilde{T}_n$ on rational period functions exploiting the action of the classical Hecke operator $T_n^\infty$ on modular integrals. Later, Choie and Zagier \cite[p. 12]{Choiezag} provides an algebraic definition of $\widetilde{T}_n$, independent of the existance of modular integrals for an arbitrary rational period function. We call an element $\widetilde{T}_n\in \mathcal{R}_n $ ``acts like the $n$-th Hecke operator on periods", if it satisfies the relation 
	\begin{align}\label{T_relation}
		(1-S)\widetilde{T}_n=T_n^\infty (1-S)+(1-T)Y
	\end{align}
for some $Y \in \mathcal{R}_n$. To describe the action of $\widetilde{T}_n$, we consider the graded ring $\mathcal{R}=\mathbb{Q}[\mathcal{M}]=\oplus_{n\in \mathbb{N}}\mathcal{R}_n$, where $\mathcal{M}=\cup_{n\in \mathbb{N}}\mathcal{M}_n$. This ring acts on the right of the vector space of meromorphic functions on $\mathbb{C}^2$ by the formula 
	\begin{align}
		\left(f \circ \sum_{i}\lambda_i\begin{bmatrix}
			a_i & b_i \\
			c_i & d_i
		\end{bmatrix}\right) (x,y)= \sum_{i} \lambda_i f(a_ix+b_iy, c_ix+d_iy).\nonumber
	\end{align} 

We next consider the function
\begin{align*}
		C(x):= \begin{cases}
			\cot(\pi x), & x \in \mathbb{C}\setminus \mathbb{Z}\\
			0, & x \in \mathbb{Z},
		\end{cases}
	\end{align*}
which is $1$-periodic, vanishes on $\Z$ and is holomorphic away from $\Z$. For the complex-valued function
 \begin{align*}
 \mathscr{C}(x,y)= C(x)C(y)+1,
 \end{align*} 
 Radchenko and Zagier \cite[p. 234, Theorem 1]{Radchenko} investigated the action of $\widetilde{T}_n$ on it to determine the multi-term functional equation of the Herglotz-Zagier function. We here introduce the complex-valued function
	\begin{align}
		\mathscr{C}^*(x,y)= C(x)-C(y),\nonumber
	\end{align}
which satisfies the three-term equation
\begin{align}
		\mathscr{C}^*(x,y)-\mathscr{C}^*(x,x+y)-\mathscr{C}^*(x+y,y)=0.\nonumber
	\end{align}
In the following result, we explore the action of $\widetilde{T}_n$ on the function $\mathscr{C}^*$.
	\begin{theorem}\label{Zagier_type_C_relation}
		Suppose that $\widetilde{T}_n\in \mathcal{R}_n$ acts like the $n$-th Hecke operators on periods. Then 
		\begin{align}
			(\mathscr{C}^*\circ \widetilde{T}_n))(x,y)-\sum_{\ell\mid n}\ell \mathscr{C}^*(\ell x, \ell y)=c(\widetilde{T}_n),\nonumber
		\end{align}
		where $c:\mathcal{R}\to \mathbb{Z}$ is the group homomorphism defined on generators by 
		\begin{align*}
			\begin{bmatrix}
				a & b \\
				c & d
			\end{bmatrix} \to \begin{cases}
				-i[\operatorname{sgn}(a+b)-\operatorname{sgn}(c+d)], & a+b\neq 0, c+d \neq 0\\
				-i[\operatorname{sgn}(a+b)-\operatorname{sgn}(d)], & a+b\neq 0, c+d= 0\\
				-i[\operatorname{sgn}(b)-\operatorname{sgn}(c+d)], & a+b= 0, c+d \neq 0.
			\end{cases}
		\end{align*}
	\end{theorem}	
In a seminal paper, Lewis and Zagier \cite{Lewis} introduced period functions associated to Maass cusp forms. For $s\in \C$ with $\Re(s)>0$ and $x \in \mathbb{C}\setminus (-\infty,0]$, a period function with spectral parameter $s$ is a holomorphic function $\psi(x)$ that   satisfies the three-term equation
\begin{align}\label{LewisZagier three term}
	\psi(x)-\psi(x+1)-\frac{1}{(x+1)^{2s}}\psi\left(\frac{x}{x+1}\right)=0,
\end{align}
together with suitable growth conditions. A period-like function with spectral parameter $s$ is a holomorphic function, that satisfies three-term equation as in \eqref{LewisZagier three term}, with or without the growth condition. The same article  presents a variety of examples of period-like functions, among which we are particularly interested in the first two.  In the second example, the authors considered the function
\begin{align}\label{lweis_zagier_second_example}
	\psi^+_s(x)={\sum_{m\ge 0}\sum_{n \ge 0}}{}^{{}^{\hspace{.1cm}*}} \  \frac{1}{(mx+n)^{2s}} \qquad \Big(\Re(s)>1, \quad  x \in \mathbb{C}\setminus (-\infty,0]\Big),
\end{align}
where $*$ means $(m,n)\neq (0,0)$ and terms with either $m$ or $n$ equal to $0$ are to be counted with multiplicity $\frac{1}{2}$.  Recently, Choie and Kumar \cite{ChoieKumar} proved that the function $\psi^+_s$ is a Hecke eigenform under the action of the operator $\widetilde{T}_n$. For $\widetilde{T}_n=\sum_{\gamma}v_{\gamma}\gamma$ with $\gamma$ being the matrix  $\gamma=\tiny{\begin{bmatrix}
		a & b \\
		c & d
	\end{bmatrix}}$ having non-negative entries, the result of Choie and Kumar precisely states that \cite[Theorem 6.2]{ChoieKumar}
\begin{align}\label{Choie_Kumar_Theorem}
		\big(\psi^+_s|_{2s}\widetilde{T}_n\big)(x)=n^{s}\sigma_{1-2s}(n)	\psi^+_s(x),
	\end{align}
	where $\sigma_s(n):=\sum_{\ell \mid n}\ell^s$ is the generalized divisor function.
	
In the first example of period-like functions, Lewis-Zagier \cite[p 228, Example 1]{Lewis} considered the function
\begin{align}\label{1st_eg}
\psi^-_s(x)= 1-x^{-2s} \qquad \quad \Big(x \in \mathbb{C}\setminus [0,\infty)\Big).
\end{align} 	
Our next result shows that the function $\psi^-_s(x)$ is a Hecke eigenform under the action of the operator $\widetilde{T}_n$. 
\begin{theorem}\label{eigenform_1st_example_lewis_zagier}
Let $\widetilde{T}_n$ acts like the $n$-th Hecke operator on periods and takes the form $\widetilde{T}_n=\sum_{\gamma}v_{\gamma}\gamma$, where 
$\gamma=\tiny{\begin{bmatrix}
		a & b \\
		c & d
	\end{bmatrix}}$ are the matrices with non-negative entries. Then for $x>0$, the function $\psi^-_s(x)$ satisfies
	\begin{align*}
		\big(\psi^-_s|_{2s}\widetilde{T}_n\big)(x)=n^{s}\sigma_{1-2s}(n)\psi^-_s(x).
	\end{align*}
\end{theorem}

We next apply the operator $\widetilde{T}_n$ on the derivative of higher order Herglotz-Zagier function to determine the corresponding multi-term functional equations. We extend the definition of the standard slash operator of weight $m$ by
	\begin{align}\label{genslash}
		(f|^{\ell}_{m}\gamma)(x):= \frac{(ad-bc)^{m/2}}{(cx+d)^m} \log^{\ell}(cx+d) f\left(\frac{ax+b}{cx+d}\right).
	\end{align}
It reduces to the standard slash operator at $\ell=0$. Our next result shows the action of  $\widetilde{T}_n$ on the function $\Phi'_k$.
\begin{theorem}\label{Multi-term-functional-equation}
	Let $\widetilde{T}_n$ acts like the $n$-th Hecke operator on periods and takes the form $\widetilde{T}_n=\sum_{\gamma}v_{\gamma}\gamma$, where 
$\gamma=\tiny{\begin{bmatrix}
		a & b \\
		c & d
	\end{bmatrix}}$ are the matrices with non-negative entries. Then for $x>0$, the derivative of the higher order Herglotz-Zagier function
	satisfies
	\begin{multline}
		\sum_{r=1}^{k} \binom{k}{r} \frac{1}{n} \left[ \bigg(\Phi'_{r}|_2^{(k-r)}\widetilde{T}_n\bigg)(x) -n\sum_{\ell \mid n}\frac{\log^{k-r}(\ell)}{\ell}\Phi'_{r}(x)\right] \nonumber\\
		= \sum_{r=1}^{k}\binom{k}{r} \tilde{\gamma}_{k-r} \left[\sum_{\gamma} v_{\gamma}\frac{\log^{r}(ax+b)}{(ax+b)(cx+d)}-\sum_{\ell \mid n}\frac{\log^{r}(\ell x)}{\ell x}\right],\nonumber
	\end{multline}
where
	\begin{align*}
		\tilde{\g}_r = \begin{cases}
			-r\g_{r-1} & 1\le r \le k\\
			1 & r=0.
		\end{cases}
	\end{align*}
\end{theorem}
Radchenko and Zagier \cite[Proposition 3]{Lewis} have shown that the sum of the matrices $ \widehat{T}_n \in \mathcal{R}_n $, given by
\begin{align*}
	\widehat{T}_n:=\sum_{\substack{0\le c <a \\ 0\le b <d \\ ad-bc=n}}\begin{bmatrix}
		a &b\\
		c & d    \end{bmatrix},
\end{align*}
acts like the $n$-th Hecke operator on periods. As an example, one can consider
\begin{align*}
	\widehat{T}_2&= \begin{bmatrix}
		1 & 0\\
		0 & 2
	\end{bmatrix}+ \begin{bmatrix}
		1 & 1\\
		0 & 2
	\end{bmatrix}+ \begin{bmatrix}
		2 & 0\\
		0 & 1
	\end{bmatrix}+ \begin{bmatrix}
		2 & 0\\
		1 & 1
	\end{bmatrix},\\
	\widehat{T}_3&=\sum_{p=0}^{2}\begin{bmatrix}
		1 & p\\
		0 & 3
	\end{bmatrix} +\sum_{p=0}^{2}\begin{bmatrix}
		3 & 0\\
		p & 1
	\end{bmatrix} +\begin{bmatrix}
		2 & 1\\
		1 & 2
	\end{bmatrix}.
\end{align*}
The next corollary shows that an application of the special case $\widehat{T}_n$ of $\widetilde{T}_n$ in Theorem \ref{Multi-term-functional-equation} determines explicit multi-term functional equations of $\Phi'_k$.
	\begin{corollary}
For  $x>0$ and any positive integer $n$, we have 
	\begin{multline}
	\sum_{r=1}^{k} \binom{k}{r} \Bigg[\sum_{\substack{0\le c <a \\ 0\le b <d \\ ad-bc=n}} \frac{\log^{k-r}(cx+d)}{(cx+d)^2}\Phi'_{r}\left(\frac{ax+b}{cx+d}\right)-\sum_{\ell \mid n} \frac{\log^{k-r}(\ell)}{\ell}\Phi'_{r}(x)\Bigg]\nonumber\\
	= \sum_{r=1}^{k}\binom{k}{r} \tilde{\gamma}_{k-r} \Bigg[\sum_{\substack{0\le c <a \\ 0\le b <d \\ ad-bc=n}}\frac{\log^{r}(ax+b)}{(ax+b)(cx+d)}-\sum_{\ell \mid n}\frac{\log^{r}(\ell x)}{\ell x}\Bigg].
\end{multline}
In particular, for $n=2$,
\begin{multline*}
	\sum_{r=1}^{k}\binom{k}{r}\bigg[\frac{\log^{k-r}(2)}{2}\left(\frac{1}{2}\Phi_{r}'\left(\frac{x}{2}\right)+\frac{1}{2}\Phi_{r}'\left(\frac{x+1}{2}\right)-\Phi'_r(x)\right)+\frac{\log^{k-r}(x+1)}{(x+1)^2}\Phi_{r}'\left(\frac{2x}{x+1}\right)\bigg]\nonumber\\
	=\Phi_{k}'(x)-\Phi_{k}'(2x)+ \sum_{r=1}^{k}\binom{k}{r}\tilde{\gamma}_{k-r} \left[\frac{\log^r(x+1)}{2(x+1)}+\frac{\log^r(2x)}{2x(x+1)}-\frac{\log^r(x)}{2x}\right].
\end{multline*}
\end{corollary}
\subsection*{Remarks} An essentially equivalent version of $\widehat{T}_n$ was earlier introduced by Merel \cite{Merel}, namely
\begin{align*}
	\widehat{T}_n^t:=\sum_{\substack{0\le c <a \\ 0\le b <d \\ ad-bc=n}}\begin{bmatrix}
		a &b\\
		c & d    \end{bmatrix}^t,
\end{align*}
which also can be considered as an example of  $\widetilde{T}_n$. An application of $\widehat{T}_n^t$ in Theorem \ref{Multi-term-functional-equation} provides another family of multi-term functional equations of $\Phi'_k$.

The paper is organized as follows. In Section 2, we study an action of $\widetilde{T}_n$ on the function $\mathscr{C}^*(x,y)$. In Section 3, we applied Theorem \ref{Zagier_type_C_relation} to establish that the function $\psi^-_s(x)$ is a Hecke eigenform under the action of $\widetilde{T}_n$. Finally, in Section 4, we utilize the fact that both the functions $\psi^+_s(x)$ and $\psi^-_s(x)$ are Hecke eigenform under the action of $\widetilde{T}_n$, to prove Theorem \ref{Multi-term-functional-equation}.

\section{Hecke action on a variant of cotangent function}
This section mainly concerns about the action of the operator $\widetilde{T}_n$ on the function $\mathscr{C}^*(x,y)$ for all $(x,y) \in \mathbb{C}^2$. We first consider a right ideal in $\mathcal{R}$, given by
 \begin{align}
	\mathscr{L}^*:=\{\xi \in \mathcal{R}:\mathscr{C}^*\circ\xi=\text{constant}\}.\nonumber
\end{align}
For $\delta: \mathbb{C} \to \{0,1\}$ denoting the characteristic function of $\mathbb{Z}$, we define a complex valued function in $\mathbb{C}^2$, given by 
\begin{align}\label{Deltatilde}
\tilde{\delta}(x,y):=\delta(y).
\end{align}
The following proposition states a necessary and sufficient condition for an element to belong to  $\mathscr{L}^*$.
\begin{proposition}\label{Zagier_type_proposition}
	An element $\xi \in \mathcal{R}$ lies in the ideal $\mathscr{L}^*$ if and only if $(1-S)\xi \in \operatorname{Ker}(\Phi)$ where $\Phi:\mathcal{R}\to \mathbb{Q}(u,v)\otimes \mathcal{V}$ is a homomorphism defined for $\gamma=\begin{bmatrix}
		a & b\\
		c & d
	\end{bmatrix}\in \mathcal{M}$ by
	\begin{align}
		\Phi(\gamma)=\frac{1}{cu+dv}\otimes \tilde{\delta}\circ \gamma,\nonumber
	\end{align}
	where $\mathcal{V}$ is the space of even complex valued function.
\end{proposition}
\begin{proof}
	The definitions of $C$ and $\delta$ together with the Taylor's expansion of the cotangent function around the origin yield as $\epsilon \to 0$,
	\begin{align}\label{Expansion of cotangent}
		C(x+\epsilon)=\frac{\delta(x)}{\pi \epsilon}+C(x)+\mathcal{O}(\epsilon),
	\end{align}
for any complex number $x$. We next check the continuity of the function $\mathscr{C}^*\circ \xi $ for some $\xi \in \mathcal{R}$ at some point $(x_0,y_0) \in \mathbb{C}^2$ and for that we set $(x,y)=(x_0+\epsilon u, y_0+\epsilon v)$, where $\epsilon$ is arbitrarily small. Now, the expansion in \eqref{Expansion of cotangent} implies 
	\begin{align}
		\mathscr{C}^*(ax+by,cx+dy)
		&= C(ax_0+by_0+\epsilon(au+bv))-C(cx_0+dy_0+\epsilon(cu+dv))\nonumber\\
		&=\frac{\delta(ax_0+by_0)}{\pi \epsilon(au+bv)}-\frac{\delta(cx_0+dy_0)}{\pi \epsilon (cu+dv)}+\mathscr{C}^*(ax_0+by_0, cx_0+dy_0)+\mathcal{O}(\epsilon).\nonumber
		\end{align}
For $\gamma=\begin{bmatrix}
		a & b\\
		c & d
	\end{bmatrix}\in \mathcal{M}$ and $S=\begin{bmatrix}
		0 & -1\\
		1 & 0
	\end{bmatrix}$, we can rewrite the above equation as	
		\begin{align}
		\mathscr{C}^*(ax+by,cx+dy)-\mathscr{C}^*(ax_0+by_0, cx_0+dy_0) 
		=-\frac{1}{\pi \epsilon}\Phi((1-S)\gamma)(x_0,y_0)+\mathcal{O}(\epsilon).\nonumber
	\end{align}
Therefore, the function $f=\mathscr{C}^*\circ \xi$ is continuous on $\mathbb{C}^2$ if and only if $(1-S)\gamma\in \operatorname{Ker}(\Phi)$.
The rest of the proof follows from an argument analogous to the one used in the proof of \cite[Proposition 2]{Radchenko}.
	\end{proof} 
	
We are now ready to prove Theorem \ref{Zagier_type_C_relation}.
\subsection{Proof of Theorem \ref{Zagier_type_C_relation}}
It follows from the definition \eqref{Deltatilde} of $\tilde{\delta}(x,y)$ that for any $Y\in \mathcal{R}_n$, we have $\tilde{\delta}\circ(1-T)Y = 0$. Therefore, the relation \eqref{T_relation} satisfied by the operator $\widetilde{T}_n$, yields
\begin{align}\label{phi-T-relation}
	\Phi((1-S)\widetilde{T}_n)=\Phi(T_n^\infty(1-S)).
\end{align}
The definition of the homomorphism $\Phi$ turns the right hand side of the above equation to
\begin{align}\label{phi-l-relation}
	\Phi(T_n^\infty(1-S))
	&= \sum_{ad=n}\sum_{0\le b <d}\left[\frac{\delta(dy)}{dv}-\frac{\delta(dx)}{du}\right]\nonumber\\
	&= \sum_{\ell \mid n}\ell \left[\frac{\delta(\ell y)}{\ell v}-\frac{\delta(\ell x)}{\ell u}\right]\nonumber\\
	&= \Phi\left(\sum_{\ell \mid n}\ell \begin{bmatrix}
		\ell & 0 \\
		0 & \ell
	\end{bmatrix}(1-S)\right).
\end{align}
Thus, \eqref{phi-T-relation} and \eqref{phi-l-relation} together imply
 \begin{align}
	(1-S)\left(\widetilde{T}_n-\sum_{\ell \mid n}\ell \begin{bmatrix}
		\ell & 0 \\
		0 & \ell
	\end{bmatrix}\right)\in \operatorname{Ker}(\Phi).\nonumber
\end{align}
Employing Proposition \ref{Zagier_type_proposition}, we can now write that for $\hat{\xi}_n = \widetilde{T}_n-\sum\limits_{\ell \mid n}\ell \begin{bmatrix}
		\ell & 0 \\
		0 & \ell
	\end{bmatrix}$,
 \begin{align}
	\mathscr{C}^*\circ \hat{\xi}_n = \text{constant}.\nonumber
\end{align}
We next concentrate in determining the value of the above constant function on the right hand side and for that we evaluate the value of the function on the left hand side at $(it, i(1+\epsilon)t)$ 
as $t\to \infty$ and $\epsilon \to 0+$. 

For any non-zero real number $\alpha$,
\begin{align}
	\lim_{t \to \infty} C(i\alpha t)=\lim_{t \to \infty} \cot (\pi i \alpha t)= \lim_{t\to \infty}i \frac{e^{-\pi \alpha t}+e^{\pi \alpha t}}{e^{-\pi \alpha t}-e^{\pi \alpha t}}= -i\operatorname{sgn}(\alpha),\nonumber
\end{align}
which concludes that for any $\g=\begin{bmatrix}
	a & b \\
	c  & d
\end{bmatrix} \in \mathcal{M}_n$, 
\begin{align}\label{Cgamma}
\lim_{t \to \infty}	(\mathscr{C}^*\circ \g)(it,i(1+\epsilon)t)
&=\lim_{t \to \infty}\left[C(it(a+b+b\epsilon))-C(it(c+d+d\epsilon))\right]\nonumber\\
&= -i\left[\operatorname{sgn}(a+b+b\epsilon)-\operatorname{sgn}(c+d+d\epsilon)\right],
\end{align}
and also for any positive integer $\ell$, 
 \begin{align}\label{Cl}
	\lim_{t \to \infty}\left(\mathscr{C}^*\circ  \begin{bmatrix}
		\ell & 0\\
		0 & \ell
	\end{bmatrix}\right) (it,(1+\epsilon)t)&=	\lim_{t \to \infty} \mathscr{C}^*(i\ell t, i (1+\epsilon)\ell t)\nonumber\\
	&=	\lim_{t \to \infty}\left[ C(i\ell t)-C(i(1+\epsilon)\ell t) \right]\nonumber\\
	&= -i\left[ \operatorname{sgn}(\ell)- \operatorname{sgn}((1+\epsilon)\ell)\right].
\end{align}
Finally, the constant map can be determined by taking the limit as $\epsilon \to 0+$ in both \eqref{Cgamma} and \eqref{Cl}. This completes the proof of the theorem.
\qed
\section{Hecke Eigenform}
In this section, we show that the function $\psi_s^-(x)$, defined in \eqref{1st_eg}, is a Hecke eigenform with respect to the Hecke operator $\widetilde{T}_n$.
\subsection{Proof of Theorem \ref{eigenform_1st_example_lewis_zagier}}
The following expression of the cotangent function 
$$\cot(x)= i \frac{1+e^{-2i x}}{1-e^{-2ix}}$$
transforms the function $\mathscr{C}^*(ixt,iyt)$, for $x$, $y >0$, into
\begin{align}
	\mathscr{C}^*(ixt,iyt)= C(ixt)-C(iyt)&=\cot(\pi i xt)-\cot(\pi i yt)\nonumber\\
	&= -2i \left(\frac{1}{e^{2\pi xt}-1}-\frac{1}{e^{2\pi y t}-1}\right).\nonumber
\end{align}
Thus, we can express the following integral as,
\begin{align}\label{C_integral_simplification}
	\frac{1}{\Gamma(2s)\z(2s)}\int_{0}^{\infty} i \mathscr{C}^*(ixt,iyt) t^{2s-1}~ dt&= 	\frac{2}{\Gamma(2s)\z(2s)}\int_{0}^{\infty}\left(\frac{1}{e^{2\pi xt}-1}-\frac{1}{e^{2\pi y t}-1}\right)t^{2s-1}~dt\nonumber\\
	&=\frac{2(2\pi)^{-2s}}{\Gamma(2s)\z(2s)}\int_{0}^{\infty}\left(\frac{1}{e^{xt}-1}-\frac{1}{e^{ y t}-1}\right)t^{2s-1}~dt\nonumber\\
	&=2(2\pi)^{-2s}\left[x^{-2s}-y^{-2s}\right],
\end{align}
where in the penultimate step we made a change of variable $t$ by $\frac{t}{2\pi}$ and the final step follows from the well-known integral representation (cf. \cite[P. 251, Theorem 12.2]{Apostol})
\begin{align}
	\Gamma(s)\zeta(s)= \int_{0}^{\infty} \frac{t^{s-1}}{e^t-1}~dt \qquad (\Re(s)>1).\nonumber
\end{align}
Therefore, \eqref{C_integral_simplification} implies 
\begin{align}\label{C_relation_psi}
	\frac{-(2\pi)^{2s}}{2\Gamma(2s)\z(2s)}\int_{0}^{\infty} i \mathscr{C}^*(ixt,iyt) t^{2s-1}~ dt=\frac{1}{y^{2s}}\left[1-\left(\frac{x}{y}\right)^{-2s}\right].
\end{align}
We next  apply $\widetilde{T}_n$ on both side of the above equation. 
Invoking Theorem \ref{Zagier_type_C_relation} and observing the fact that $c(\widetilde{T}_n)=0$, as all its entries are non-negative by assumption, the left hand side of \eqref{C_relation_psi} reduces to
\begin{align}\label{lhs}
	\frac{-(2\pi)^{2s}}{2\Gamma(2s)\z(2s)}\int_{0}^{\infty} i \mathscr{C}^*(ixt,iyt) \circ \widetilde{T}_n~ t^{2s-1}~ dt
	&=\frac{-(2\pi)^{2s}i}{2\Gamma(2s)\z(2s)}\int_{0}^{\infty} 
	\sum_{\ell \mid n} \ell \mathscr{C}^*(i\ell xt,i\ell yt)~t^{2s-1}~ dt\nonumber\\
	&= \sum_{\ell \mid n} \ell \frac{1}{(\ell y)^{2s}}\left[1-\left(\frac{x}{y}\right)^{-2s}\right]\nonumber\\
	&= \sigma_{1-2s}(n)\frac{1}{y^{2s}}\left[1-\left(\frac{x}{y}\right)^{-2s}\right].
\end{align}
where the penultimate step follows from \eqref{C_relation_psi}.
Applying $\widetilde{T}_n$ on the right hand side of \eqref{C_relation_psi}, we obtain
\begin{align}\label{rhs}
	\left(\frac{1}{y^{2s}}\left[1-\left(\frac{x}{y}\right)^{-2s}\right]\right)\circ \widetilde{T}_n
	&=\left(\frac{1}{y^{2s}}\left[1-\left(\frac{x}{y}\right)^{-2s}\right]\right)\circ \sum_{\g}v_{\g}\g\nonumber\\
	&= \sum_{\g} \frac{v_{\g}}{(cx+dy)^{2s}} \left[1-\left(\frac{ax+by}{cx+dy}\right)^{-2s}\right].
\end{align}
Thus, \eqref{lhs} and \eqref{rhs} together yield
\begin{align}
	\sum_{\g} \frac{v_{\g}}{(cx+dy)^{2s}} \left[1-\left(\frac{ax+by}{cx+dy}\right)^{-2s}\right]=\sigma_{1-2s}(n)\frac{1}{y^{2s}}\left[1-\left(\frac{x}{y}\right)^{-2s}\right].\nonumber
\end{align}
Finally, we substitute $y=1$ and  multiply $n^s$ on the both sides of the above equation to conclude the proof of our result.
\qed
\section{Hecke-type action on $\Phi'_{k}(x)$}\label{Main Results}
In this section, we first relate the function $\Phi'_{k}(x)$ with the function $\psi^+_s(x)$ with spectral parameter $\frac{s+1}{2}$ and then take an advantage of the fact that both the functions $\psi^+_s(x)$ and $\psi^-_s(x)$ are Hecke eigenforms with respect to the operator $\widetilde{T}_n$, to establish Theorem \ref{Multi-term-functional-equation}.

Dixit et al. \cite[p. 7, Equation (1.15)]{DixitMTZ} introduced a new generalization of the Herglotz-Zagier function, defined for $\Re(s)>0, \,  s \neq 1$ and $x \in \mathbb{C}\setminus (-\infty,0]$, namely
\begin{align}
	\Phi(s,x)= \sum_{n=1}^{\infty} \frac{1}{n} \bigg[\zeta(s,nx)-\frac{(nx)^{1-s}}{s-1}\bigg],\nonumber
\end{align}
where $\zeta(s,x)$ denotes the Hurwitz zeta function, given by $\zeta(s,x)=\sum_{m=0}^{\infty}\frac{1}{(m+x)^s}$. The authors termed this function as Herglotz-Hurwitz function. In the same article, the authors established an elegant connection between $\Phi(s,x)$ and the $k$-th order Herglotz-Zagier function $\Phi_{k}(x)$, which we state in the following lemma.
  \begin{lemma}\label{Dixit's-Lemma}
   Let $k$ be any positive integer. We have
\begin{align}
	\lim_{s \to 1} \frac{\partial^{k-1}}{\partial s^{k-1}}\Phi(s,x)=\frac{(-1)^{k}}{k}\Phi_{k}(x).\nonumber
\end{align}
\end{lemma}

Let $ \Phi_x\left(s,x_0\right)$ be the value of the partial derivative of $\Phi(s,x)$ with respect to $x$ at the point $x_0$.  The following lemma relates $ \Phi_x\left(s,x\right)$ with Lewis-Zagier's second example  $\psi^+_s(x)$, defined in \eqref{lweis_zagier_second_example}.
\begin{lemma}\label{F-relation-derivative-Phi}
	For $x>0$, we have
	\begin{align}
		\Phi_x\left(s,x\right)= x^{-s}\zeta(s)+\frac{1}{2}\left(1 - x^{-(s+1)}\right)s\zeta(s+1) -s\psi^+_{\frac{s+1}{2}}(x).\nonumber
	\end{align}
\end{lemma}
\begin{proof} It follows from the definition of $\Phi(s,x)$ that,
\begin{align}
	\Phi_x\left(s,x\right)&= \sum_{n=1}^{\infty}\sum_{m=0}^{\infty} \frac{-s}{(m+nx)^{s+1}}+x^{-s}\z(s)\nonumber\\
	&= \sum_{n=1}^{\infty}\sum_{m=1}^{\infty} \frac{-s}{(m+nx)^{s+1}}-x^{-(s+1)}s\zeta(s+1)+x^{-s}\z(s)\nonumber\\
		&=-s\left[\psi^+_{\frac{s+1}{2}}(x) - \frac{1}{2}x^{-(s+1)}\zeta(s+1) - \frac{1}{2}\zeta(s+1)\right]-x^{-(s+1)}s\zeta(s+1)+x^{-s}\z(s).\nonumber
\end{align}
Finally, after simplification, we arrive at our conclusion. 
\end{proof}

We next show that the function $q(z) = \frac{1}{z}$ for $z>0$ is an Hecke eigenform with respect to the operator $\widetilde{T}_n$. 
\begin{lemma}\label{Pole_Cancellation}
		Let $\widetilde{T}_n$ acts like the $n$-th Hecke operator on periods and takes the form $\widetilde{T}_n=\sum_{\gamma}v_{\gamma}\gamma$, where 
$\gamma=\tiny{\begin{bmatrix}
		a & b \\
		c & d
	\end{bmatrix}}$ are the matrices with non-negative entries. Then for $z \in \mathcal{H}$, the following holds:
	\begin{align}
		\left(\frac{1}{z}\right)|_2\widetilde{T}_n =  \sigma_1(n)\left(\frac{1}{z}\right).\nonumber
	\end{align}
\end{lemma}
\begin{proof}
We denote the Eisenstein series of weight $2$ by $G_2(z)$, which is defined for every $z\in \mathcal{H}$ by the following double sum 
	\begin{align}
		G_2(z)= {\sum_{m,n \in \Z}}{}^{{}^{\hspace{.1cm}'}} \  \frac{1}{(mz+n)^{2}},\nonumber
	\end{align}
where $'$ on the summation sign means that the term $m = n = 0$ is to be omitted. It is a modular integral of weight $2$, since 
it satisfies
	\begin{align}\label{G_relation}
		G_2|_2T(z)=G_2(z)\qquad \text{and}\qquad G_2|_2 S(z)=G_2(z) -\frac{2\pi i}{z}.
	\end{align}
The relation \eqref{T_relation}, satisfied by the operator $\widetilde{T}_n$, yields
	\begin{align}
		G_2|_2(1-S)\widetilde{T}_n=G_2|_2T_n^{\infty}(1-S)+G_2|_2(1-T)Y.\nonumber
	\end{align}
Therefore, by applying \eqref{G_relation}, the above equation reduces to 
	\begin{align}\label{rational_function_relation}
		\left(\frac{2\pi i}{z}\right)|_2\widetilde{T}_n=G_2|_2T_n^{\infty}(1-S).
	\end{align}
Employing the fact [cf. \cite[p. 54]{Knopp1}, \cite[Proposition 13]{Serre}] that the function $G_2(z)$ is a Hecke eigen form with respect to the standard Hecke operator $T_n^{\infty}$, corresponding to the eigen value $\sigma_1(n)$, we can transform \eqref{rational_function_relation} as 
	\begin{align}
			\left(\frac{2\pi i}{z}\right)|_2\widetilde{T}_n&=\sigma_1(n)G_2|_2(1-S)\nonumber\\
			&= \sigma_1(n)\left(\frac{2\pi i}{z}\right),\nonumber
	\end{align}
where the last step follows from \eqref{G_relation}. Finally, multiplying $(2\pi i)^{-1}$ on the both side of the above equation, we conclude our result.
\end{proof}

We are now ready to prove Theorem \ref{Multi-term-functional-equation}. 
\subsection{Proof of Theorem \ref{Multi-term-functional-equation}}
The action of the standard slash operator on the function $\psi_s^+$ with spectral parameter $\frac{s+1}{2}$ derives \eqref{Choie_Kumar_Theorem} as
\begin{align}
			\sum_{\gamma} v_{\gamma} \frac{1}{(cx+d)^{s+1}} \psi^+_{\frac{s+1}{2}}\left(\frac{ax+b}{cx+d}\right)=	\sum_{\ell \mid n}\frac{1}{\ell^s} \psi^+_{\frac{s+1}{2}}(x).\nonumber
		\end{align}
Thus, it follows from Lemma \ref{F-relation-derivative-Phi} that
\begin{align}\label{main_eq_after_F_s}
\sum_{\gamma} v_{\gamma}&\frac{1}{(cx+d)^{s+1}} \Phi_x\left(s, \frac{ax+b}{cx+d}\right) - \sum_{\ell\mid n}\frac{1}{\ell^s} \Phi_x(s,x) = \sum_{\gamma}v_{\gamma}\frac{\z(s)}{(ax+b)^s (cx+d)}-\sum_{\ell \mid n} \frac{\z(s)}{(\ell x)^s}\nonumber\\
&\, \, +  \frac{s\zeta(s+1)}{2}\left(\sum_{\gamma} \frac{v_{\gamma}}{(cx+d)^{s+1}}\left(1-\left(\frac{ax+b}{cx+d}\right)^{-(s+1)}\right)-\sum_{\ell \mid n} \frac{1}{\ell^{s}}\left(1-x^{-(s+1)}\right)\right).
\end{align}
Expanding the standard slash action on the function $\psi^-_{s}(x)$ with spectral parameter $\frac{s+1}{2}$ in Theorem \ref{eigenform_1st_example_lewis_zagier}, we obtain
\begin{align}
	\sum_{\gamma} \frac{v_{\gamma}}{(cx+d)^{s+1}}\left(1-\left(\frac{ax+b}{cx+d}\right)^{-(s+1)}\right)=\sum_{\ell \mid n} \frac{1}{\ell^{s}}\left(1-x^{-(s+1)}\right).\nonumber
\end{align}
Thus, by applying the above identity, we can reduce \eqref{main_eq_after_F_s} as 
\begin{align}\label{Reduced_equation_after_F_s}
		\sum_{\gamma} v_{\gamma}\frac{1}{(cx+d)^{s+1}} \Phi_x\left(s, \frac{ax+b}{cx+d}\right) - \sum_{\ell\mid n}\frac{1}{\ell^s} \Phi_x(s,x)=\sum_{\gamma}v_{\gamma}\frac{\z(s)}{(ax+b)^s (cx+d)}-\sum_{\ell \mid n} \frac{\z(s)}{(\ell x)^s}.
\end{align}
Now, we differentiate both side of the above equation $k-1$ times and then evaluate the limit as $s \to 1$. The left-hand side of \eqref{Reduced_equation_after_F_s} transforms to
\begin{align}\label{main_LHS}
&\lim_{s \to 1} \bigg[\sum_{\gamma} v_{\gamma}\frac{\partial^{k-1}}{\partial s^{k-1}}\left[ \frac{1}{(cx+d)^{s+1}} \Phi_x\left(s, \frac{ax+b}{cx+d}\right) \right] - \sum_{\ell\mid n} \frac{\partial^{k-1}}{\partial s^{k-1}}\left[ \frac{1}{\ell^s} \Phi_x(s,x)\right]\bigg]\nonumber\\
	&=\lim_{s \to 1}\Bigg[ \sum_{r=0}^{k-1}\binom{k-1}{r}\bigg[\sum_{\gamma}v_{\gamma} \frac{(-1)^{k-1-r}\log^{k-1-r}(cx+d)}{(cx+d)^{s+1}}\frac{\partial^r}{\partial s^r}\left(\Phi_x\left(s,\frac{ax+b}{cx+d}
	\right)\right)\nonumber\\
	&\hspace{5.5cm}-\sum_{\ell \mid n} \frac{(-1)^{k-1-r}\log^{k-1-r}(\ell)}{\ell^s}\frac{\partial^r}{\partial s^r}\left(\Phi_x\left(s,x
	\right)\right)\bigg]\Bigg]\nonumber\\
	&=\sum_{r=0}^{k-1} \binom{k-1}{r} \frac{(-1)^{k}}{r+1} \left[\sum_{\gamma} v_{\gamma} \frac{\log^{k-1-r}(cx+d)}{(cx+d)^2}\Phi'_{r+1}\left(\frac{ax+b}{cx+d}\right)-\sum_{\ell \mid n} \frac{\log^{k-1-r}(\ell)}{\ell}\Phi'_{r+1}(x)\right],
\end{align}
where in the penultimate step, we have applied Leibnitz rule of differentiation and the last step follows from Lemma \ref{Dixit's-Lemma}.

Employing Lemma \ref{Pole_Cancellation} along with the {\em identity theorem} for analytic functions, we obtain that for $x>0$
\begin{align}\label{Cancelled pole}
	\sum_{\g} \frac{v_{\g}}{(ax+b)(cx+d)}=\sum_{\ell\mid n} \frac{1}{\ell x},
\end{align}
which turns the right hand side of \eqref{Reduced_equation_after_F_s} into 
\begin{align}\label{Main_RHS}
	 \lim_{s \to 1}\Bigg[\sum_{\gamma}v_{\gamma}\frac{1}{cx+d}\frac{\partial^{k-1}}{\partial s^{k-1}}\left(\frac{\z(s)}{(ax+b)^s}\right)-\sum_{\ell \mid n} \frac{\partial^{k-1}}{\partial s^{k-1}}\left(\frac{\z(s)}{(\ell x)^s}\right)\Bigg] = L_1-L_2,
\end{align}
		where 
		\begin{align}
		&L_1:= \lim_{s \to 1}\sum_{\gamma}v_{\gamma}\frac{1}{cx+d}\left[\frac{\partial^{k-1}}{\partial s^{k-1}}\left(\frac{\z(s)}{(ax+b)^s}\right)-\frac{(-1)^{k-1}(k-1)!(ax+b)^{-1}}{(s-1)^{k}}\right],\nonumber\\
		&L_2:=\lim_{s \to 1}\sum_{\ell \mid n}\left[ \frac{\partial^{k-1}}{\partial s^{k-1}}\left(\frac{\z(s)}{(\ell x)^s}\right)-\frac{(-1)^{k-1}(k-1)!(\ell x)^{-1}}{(s-1)^{k}}\right].\nonumber
		\end{align}
We next concentrate in evaluating the limits $L_1$ and $L_2$, which is based on the Cauchy product of the corresponding Laurent series at $s=1$ for the functions $\zeta(s)$ and $y^{-s}$ for $y>0$. The  Laurent series expansions
		 \begin{align*}
			\zeta(s)=\frac{1}{s-1}+\sum_{n=0}^{\infty} \frac{(-1)^n\gamma_n}{n!} (s-1)^n , \qquad
			y^{-s}=\frac{1}{y}\sum_{n=0}^{\infty} \frac{(-1)^n\log^n(y)}{n!} (s-1)^n
		\end{align*}
evaluate the following limit as
			\begin{multline}
			\lim_{s \to 1}\bigg[\frac{\partial^{k-1}}{\partial s^{k-1}}\left(	y^{-s}\zeta(s)\right)-\frac{(-1)^{k-1} (k-1)!\, y^{-1}}{(s-1)^{k}}\bigg]\\
			=\frac{(-1)^{k}\log^{k}(y)}{ky}+\frac{(-1)^{k-1}}{y}\sum_{r=0}^{k-1}\binom{k-1}{r} \gamma_{k-1-r} \log^{r}(y).\nonumber		
		\end{multline}
Thus, we can write
		\begin{align}
		&	L_1= \sum_{\gamma}v_{\gamma}\bigg[\frac{(-1)^{k}\log^{k}(ax+b)}{k(ax+b)(cx+d)}+\frac{(-1)^{k-1}}{(ax+b)(cx+d)}\sum_{r=0}^{k-1}\binom{k-1}{r}\gamma_{k-1-r}\log^{r}(ax+b)\bigg],\nonumber\\
		& L_2= \sum_{\ell \mid n} \bigg[\frac{(-1)^{k}\log^{k}(\ell x)}{k\ell x}+\frac{(-1)^{k-1}}{\ell x}\sum_{r=0}^{k-1}\binom{k-1}{r}\gamma_{k-1-r}\log^{r}(\ell x)\bigg].\nonumber
		\end{align}
Inserting the above limits $L_1$ and $L_2$ into \eqref{Main_RHS} and then applying \eqref{Reduced_equation_after_F_s}, we equate the equations \eqref{main_LHS} and \eqref{Main_RHS} to obtain 
		\begin{multline}
			\sum_{r=0}^{k-1} \binom{k-1}{r} \frac{1}{r+1} \left[\sum_{\gamma} v_{\gamma} \frac{\log^{k-1-r}(cx+d)}{(cx+d)^2}\Phi'_{r+1}\left(\frac{ax+b}{cx+d}\right)-\sum_{\ell \mid n} \frac{\log^{k-1-r}(\ell)}{\ell}\Phi'_{r+1}(x)\right]\nonumber\\
			= -\sum_{r=0}^{k-1}\binom{k-1}{r} \gamma_{k-1-r} \left[\sum_{\gamma} v_{\gamma}\frac{\log^{r}(ax+b)}{(ax+b)(cx+d)}-\sum_{\ell \mid n}\frac{\log^{r}(\ell x)}{\ell x}\right]\nonumber\\
			+\frac{1}{k}\Bigg[\sum_{\gamma} v_{\gamma}\frac{\log^{k}(ax+b)}{(ax+b)(cx+d)}-\sum_{\ell \mid n}\frac{\log^{k}(\ell x)}{\ell x}\Bigg].
		\end{multline}
Note that, the term corresponding to $r=0$ of the first sum on the right hand side of the above equation vanishes, which follows from \eqref{Cancelled pole}. Multiplying both side by $k$ and then substituting $r$ by $r-1$ on the left hand side, we can write
				\begin{multline}
				\sum_{r=1}^{k} \binom{k}{r} \left[\sum_{\gamma} v_{\gamma} \frac{\log^{k-r}(cx+d)}{(cx+d)^2}\Phi'_{r}\left(\frac{ax+b}{cx+d}\right)-\sum_{\ell \mid n} \frac{\log^{k-r}(\ell)}{\ell}\Phi'_{r}(x)\right]\\
				= -\sum_{r=1}^{k-1}\binom{k}{r} (k-r) \gamma_{k-1-r} \left[\sum_{\gamma} v_{\gamma}\frac{\log^{r}(ax+b)}{(ax+b)(cx+d)}-\sum_{\ell \mid n}\frac{\log^{r}(\ell x)}{\ell x}\right]\\
				+\Bigg[\sum_{\gamma} v_{\gamma}\frac{\log^{k}(ax+b)}{(ax+b)(cx+d)}-\sum_{\ell \mid n}\frac{\log^{k}(\ell x)}{\ell x}\Bigg].\nonumber
			\end{multline}	
Finally, an application of the general slash operator, defined in \eqref{genslash}, leads to the conclusion. This completes the proof of the theorem.
\qed	 	 
\subsection*{Acknowledgements} The first author’s research was partially supported by Anusandhan National Research
		 	 Foundation (ANRF) grant ANRF/ARGM/2025/000175/MTR of Govt. of India and the Cumulative Professional Development Allowance (CPDA) grant from the affiliated institute. The second author is currently a Ph.D student at IIT Kharagpur and her research was supported by University Grants Commision (UGC), Govt. of India.

\end{document}